\documentclass[12pt,a4paper]{article}
\usepackage{amsthm,amsmath,amssymb,
extsizes,MnSymbol,easymat,mathrsfs,etex}

\usepackage[active]{srcltx}
\newcommand{\matt}[1]{\left[\begin{smallmatrix}
   #1
  \end{smallmatrix}\right]}
\newcommand{\mat}[1]{\begin{bmatrix}
   #1
  \end{bmatrix}}

\newcommand{\HH}{\mathbb H}
\newcommand{\CC}{\mathbb C}
\newcommand{\RR}{\mathbb R}

\newtheorem{theorem}{Theorem}
\newtheorem{lemma}{Lemma}

\theoremstyle{definition}

\theoremstyle{remark}

\begin{document}
\title{Roth's solvability criteria for the matrix equations
${AX-\widehat XB=C}$ and ${X-A\widehat{X}B=C}$ over the skew field of quaternions with an involutive automorphism $q\mapsto \hat q$
\thanks{Published in  Linear Algebra Appl. 510 (2016) 246--258.}}

\date{}

\author{Vyacheslav Futorny\thanks{Department of Mathematics, University of S\~ao Paulo, Brazil, {futorny@ime.usp.br}}
\and
Tetiana Klymchuk
\thanks{Universitat Polit\`{e}cnica de Catalunya, Barcelona, Spain, {tetiana.klymchuk@upc.edu}}
\and
Vladimir V.
Sergeichuk\thanks{Institute of Mathematics,
Kiev, Ukraine, {sergeich@imath.kiev.ua}}
}

\maketitle

\begin{abstract}
The matrix equation $AX-XB=C$  has a solution if and only if the matrices $\matt{A&C\\0&B}$ and $\matt{A &0\\0 & B}$ are similar. This criterion was proved over a field by W.E.~Roth (1952) and over the skew field of quaternions by Huang Liping (1996). H.K.~Wimmer (1988) proved that the matrix equation $X-AXB=C$ over a field has a solution if and only if the matrices $\matt{A & C \\0 &I}$ and $\matt{I & 0 \\0 & B}$ are simultaneously equivalent to $\matt{A & 0 \\0 & I}$ and $\matt{I & 0 \\
0 & B}$.
We extend these criteria to the matrix equations $AX-\widehat XB=C$ and $X-A\widehat XB=C$ over the skew field of quaternions with a fixed involutive automorphism $q\mapsto \hat q$.


\end{abstract}
\section{Introduction}

Let $\HH$ be the skew field of quaternions with a fixed involutive automorphism $h\mapsto \hat h $; that is, a bijection $\HH\to\HH$ (possibly, the identity) such that
\[
\widehat{h+k}=\hat{h}+\hat{k},\quad
\widehat{hk}=\hat{h}\hat{k},\quad \Hat{\Hat h}=h
\]
for all $h,k\in \HH$. If $H=[h_{ij}]$ is a quaternion matrix, then we write $\widehat H:=[\hat h_{ij}]$. We prove two criteria of solvability of quaternion matrix equations (see Theorems \ref{lih} and \ref{iii2}):
\begin{itemize}
  \item
$AX-\widehat XB=C$ has a solution if and only if
$\widehat S^{-1}\matt{A & C\\0 & B}S=\matt{A & 0 \\0 & B}$ for some nonsingular $S$;
  \item
$X-A\widehat XB=C$ has a solution if and only if
$\matt{A&C\\0&I}R=\widehat{S}\matt{A&0\\0&I}$ and
$\matt{I&0\\0&B}R=S\matt{I&0\\0&B}$ for some nonsingular $S$ and $R$.
\end{itemize}
In order to prove them, we represent these quaternion matrix equations by complex matrix equations using the injective homomorphism \begin{equation}\label{pmt}
a+bi+cj+dk\mapsto\mat{a+bi&c+di\\-c+di&a-bi}
\end{equation}
of $\HH$ to the matrix ring $\CC^{2\times 2}$, and then we use known criteria of solvability of complex matrix equations.

In Sections \ref{sse} and \ref{ssd}, we give a brief exposition of some results on Roth theorems and their generalizations.

\subsection{Roth's theorems over a field}
\label{sse}

Roth \cite{roth}
proved two criteria of solvability of matrix equations over a field:
\begin{equation}\label{roh1}
\parbox[c]{0.8\textwidth}{$AX-YB=C$ has a solution if and only if
$\left[\begin{smallmatrix}
  A & C \\
  0 & B \\
\end{smallmatrix}\right]$ and
$\left[\begin{smallmatrix}
  A & 0 \\
  0 & B \\
\end{smallmatrix}\right]
$ are equivalent (i.e., have equal rank),
}
\end{equation}
and
\begin{equation}\label{roh2}
\parbox[c]{0.8\textwidth}{$AX-XB=C$ has a solution if and only if
$\matt{A & C\\0 & B}$ and
$\matt{A & 0 \\0 & B}$ are similar.
}
\end{equation}
The criterion \eqref{roh2} is known as
\emph{Roth's removal rule}.

Wimmer \cite{wimm} (see also Yusun \cite{yus}) proved that
\begin{equation}\label{nrg}
\parbox[c]{0.8\textwidth}{$X-AXB=C$ over a field has a solution if and only if
$\matt{A & C \\0 &I}R=S\matt{A & 0 \\0 & I}$ and $\matt{I & 0 \\0 & B}R=S\matt{I & 0 \\0 & B}$
for some nonsingular $S$ and $R$.
}
\end{equation}

The necessities in \eqref{roh1}--\eqref{nrg} are clear; for example, if $X$ is a solution of $AX-XB=C$, then
\begin{equation*}\label{gfl}
\begin{bmatrix}
    I&-X\\0&I
  \end{bmatrix}
  \begin{bmatrix}
    A&0\\0&B
  \end{bmatrix}
  \begin{bmatrix}
    I&X\\0&I
  \end{bmatrix} =
  \begin{bmatrix}
    A&AX-XB\\0&B
  \end{bmatrix}
  = \begin{bmatrix}
    A&C\\0&B
  \end{bmatrix};
\end{equation*}
see also \eqref{dss}.
The sufficiencies  in \eqref{roh1}--\eqref{nrg}  are surprising: if
$\left[\begin{smallmatrix}
  A & C \\
  0 & B \\
\end{smallmatrix}\right]$ and
$\left[\begin{smallmatrix}
  A & 0 \\
  0 & B \\
\end{smallmatrix}\right]
$ are equivalent or similar, then the transforming matrices can be taken to be upper triangular.

W.E. Roth proved \eqref{roh1} and \eqref{roh2} using canonical forms.  Flanders and Wimmer \cite{fla} gave  invariant proofs, which are presented in the books \cite[Theorem 44.3]{pra} and \cite[Section 12.5]{lan}.  Other proofs of Roth's theorems were given by R.~Feinberg (1975), J.K.~Baksalary and R.~Kala (1979),  R.~Hartwig (1983), Jiong Sheng Li (1984)  A.J.B.~Ward (1993, 1999), F.~Gerrish and A.J.B~Ward (1998, 2000), Yu.A~Al'pin and
S.N.~Il'in (2006), and M.~Lin and H.K. Harald (2011). Guralnick \cite{gur,gur1} and  Gustafson \cite{gus} extended Roth's theorems to matrices and sets of matrices over commutative rings. Rosenblum \cite{ros} showed that Roth's  theorem  is not in general valid  for bounded  operators on a Hilbert space, but it is valid  for selfadjoint operators. Fuhrmann and Helmke \cite{f-h} pointed out that Roth's theorem \eqref{roh2}  is also about existence of complementary subspaces.

Statements \eqref{roh1}--\eqref{nrg} are the most elegant criteria for existence of solutions of $AX-YB=C$, $AX-XB=C$, and $X-AXB=C$, though one can write each of these matrix equations as a system of linear equations $Mx=c$ and formulate criteria for existence and uniqueness of solutions via $M$ and $c$; see \cite[Section 12.3]{lan} and \cite{kuc}.  However, the obtained conditions are not convenient since the system of linear equations $Mx=c$ is large and can be ill-conditioned.
Note that the complex matrix equation $AX-XB=C$ has a unique solution if and only if $A$ and $B$ have no common eigenvalues and the complex matrix equation $X-AXB=C$ has a unique solution if and only if $\lambda \mu \ne 1$ for all eigenvalues $\lambda$ of $A$ and $\mu $ of $B$; see \cite[Section 12.3]{lan}.

Dmytryshyn and K{\aa}gstr{\"o}m \cite[Theorem 6.1]{dmy} extended Roth's criterions to the systems
\[A_iX_{i'}M_i-N_iX_{i''}^{\varepsilon_i} B_i=C_i,\qquad
i=1,\dots,s\]
of matrix equations with unknown matrices  $X_1,\dots,X_t$ over a field $\mathbb F$ of characteristic not 2, in which
all $i',i''\in\{1,\dots,t\}$ and each $X_{i''}^{\varepsilon_i}$ is either $X_{i''}$, or $X_{i''}^T$, or  $X^*_{i''}$ (if $\mathbb F=\mathbb C$).

Bevis,  Hall, and Hartwig \cite{bev,bev1} proved that
\begin{equation}\label{rod}
\parbox[c]{0.8\textwidth}{$AX-\bar XB=C$ over $\CC$ (in which $\bar X$ is the complex conjugate of $X$) has a solution if and only if $\bar S^{-1}\matt{A & C \\0 & B}S=\matt{A & 0 \\0 & B}$ for some nonsingular $S$.
}
\end{equation}
The theory of equations of the form $AXM-N\bar XB=C$ is summarized in Wu and Zhang's book \cite{Wu}.

\subsection{Quaternion matrix equations}
\label{ssd}

The quaternion matrix equations $AX-XB=C$ and $X+AXB=C$ are studied in L. Rodman's book \cite[Section 5.11]{rod}.
Solutions of $AX-XB=C$ are analyzed by Bolotnikov \cite{bol}.
Liping \cite{lip}\footnote{Huang Liping also publishes as Liping  Huang, Li-Ping Huang, and Li Ping Huang.} studies the quaternion matrix equation $AXB+CXD=E$.

Liping \cite[Corollary 3]{lip} proved that \begin{equation}\label{rong}
\parbox[c]{0.8\textwidth}{$AX-XB=C$ over $\HH$ has a solution if and only if
$\left[\begin{smallmatrix}
  A & C \\
  0 & B \\
\end{smallmatrix}\right]$ and
$\left[\begin{smallmatrix}
  A & 0 \\
  0 & B \\
\end{smallmatrix}\right]
$ are similar.
}
\end{equation}
Our proofs of Theorems \ref{lih} and \ref{iii2} are close to his proof;
we extend \eqref{roh2}--\eqref{rong} to the matrix equations
$AX-\widehat XB=C$ and $X-A\widehat XB=C$
over $\HH$ with a fixed involutive automorphism $q\mapsto \hat q$, which can be the identity.

Most authors study matrix equations over $\HH$ with the identity automorphism or the involutive automorphism
\begin{equation}\label{sxw}
h=a+bi+cj+dk\quad\mapsto\quad \tilde h:=j^{-1}hj=a-bi+cj-dk.
\end{equation}
Jiang and Ling \cite[Theorem 3.2]{jia2} proved that
$AX-\widetilde XB=C$ over $\HH$ with the automorphism \eqref{sxw} has a solution if and only if
$\left[\begin{smallmatrix}
  A_{\sigma } & C_{\sigma } \\
  0 & B_{\sigma } \\
\end{smallmatrix}\right]$ and
$\left[\begin{smallmatrix}
  A_{\sigma } & 0 \\
  0 & B_{\sigma } \\
\end{smallmatrix}\right]
$ are similar over $\RR$,
where
\[
A_{\sigma }:=\begin{bmatrix}
A_1 & A_2&-A_3&A_4 \\
A_2 & -A_1&-A_4&-A_3\\
A_3 & -A_4&A_1&A_2 \\
A_4 & A_3&A_2&-A_1 \\
\end{bmatrix}\in\RR^{4n\times 4n}
\]
is the real representation of a quaternion matrix $A=A_1+A_2i+A_3j+A_4k\in\HH^{n\times n}$. Jiang and Wei \cite{jia1, jia} obtained expressions for exact solutions of $X - AXB = C$ and $X-A \widetilde XB=C$ in terms of the coefficients of characteristic polynomials; explicit solutions of these equations were also obtained by Song, Chen, and Liu \cite{son1,son}.
Yuan and Liao \cite{yua} studied $X - A \widetilde XB = C$ using the complex representation of quaternion matrices.

\subsection{Involutive automorphisms of $\HH$}

Klimchuk and Sergeichuk \cite[Lemma 1]{kli1} proved the following lemma.

\begin{lemma}\label{lih2}
Each nonidentical involutive
automorphism of\/ $\HH$ has the form \begin{equation}\label{cwo}
h=a+bi+cj+dk\quad \mapsto\quad i^{-1}hi=a+bi-cj-dk
\end{equation}
in a suitable set of orthogonal
imaginary units $i,j,k\in\HH$.
\end{lemma}

Two advantages of the automorphism \eqref{cwo} over \eqref{sxw}, which will be used in the next sections, are indicated in \cite{kli1}:
\begin{itemize}
  \item
If $h\mapsto \hat h$ is \eqref{cwo} and $h\in\HH$ is represented in the form $h=u+vj$ with $u,v\in\CC$, then $\hat h=u-vj$; compare with $\tilde h=\bar u+\bar vj$. By Lemma \ref{lih2}, each involutive automorphism has the form
\begin{equation}\label{kli}
a+bi+cj+dk\quad\mapsto\quad a+bi+\varepsilon (cj+dk)
\end{equation}
for some $\varepsilon \in\{1,-1\}$,
up to reselection of the orthogonal
imaginary units $i,j,k$.
This admits to study equations
over $\HH$ with the identity automorphism and with \eqref{cwo} simultaneously; see \cite[Section 3]{kli1} and the proofs of Theorems \ref{lih} and \ref{iii2}.

  \item
If $h\mapsto\hat h$ is \eqref{kli}, then
\begin{equation}\label{nvn}
\parbox[c]{0.8\textwidth}{each square quaternion matrix is ${\wedge}$-similar to a complex matrix (two quaternion matrices $A$ and $B$ are said to be \emph{${\wedge}$-similar} if $\widehat S^{-1}AS=B$ for some nonsingular quaternion matrix $S$).
}
\end{equation}
\end{itemize}
A canonical form of a quaternion matrix \begin{itemize}
  \item[(a)] under similarity was given by Wiegmann \cite{wie} (see also \cite{zha} and \cite[Theorem 5.5.3]{rod}),
  \item[(b)] under ${\wedge}$-similarity with $h\mapsto \hat h$ defined in \eqref{cwo} was given in \cite[Theorem 3]{kli1},
  \item[(c)] under ${\wedge}$-similarity with $h\mapsto \hat h$ defined in \eqref{sxw} was given by Liping \cite[Theorem 3]{hua}.
\end{itemize}
The canonical forms (a) and (b) (as distinct from (c)) are complex matrices, which ensures \eqref{nvn}.

\section{Roth's theorem for the quaternion matrix equation $AX-\widehat XB=C$}

The following theorem generalizes \eqref{roh2} and \eqref{rong}.

\begin{theorem}\label{lih}
Let $h\mapsto \hat h$ be an involutive automorphism of the skew field of quaternions.
The quaternion matrix equation
${AX-\widehat XB=C}$ has a solution if and only if
\begin{equation}\label{tph}
\widehat S^{-1}\begin{bmatrix}
    A&C\\0&B
  \end{bmatrix}S= \begin{bmatrix}
    A&0\\0&B
  \end{bmatrix}
\end{equation}
for some nonsingular $S$.
\end{theorem}

\begin{proof}
$\Longrightarrow $. \ If $X$ is a solution of ${AX-\widehat XB=C}$, then \eqref{tph} holds with $S=\matt{ I&-X\\0&I}$ since
\[
\begin{bmatrix}
    I&-\widehat X\\0&I
  \end{bmatrix}
  \begin{bmatrix}
    A&0\\0&B
  \end{bmatrix}
  \begin{bmatrix}
    I&X\\0&I
  \end{bmatrix} =
  \begin{bmatrix}
    A&AX-\widehat XB\\0&B
  \end{bmatrix}
  = \begin{bmatrix}
    A&C\\0&B
  \end{bmatrix}.
\]

$ \Longleftarrow$. \
Due to Lemma \ref{lih2}, we suppose that the automorphism $h\mapsto\hat h$ is of the form \eqref{kli}. Let \eqref{tph} hold.
\medskip

\emph{Case 1: $A$ and $B$ are complex matrices}.
Write
\[
  C=C_1+C_2j,\quad
  S=S_1+S_2j,
\]
in which $C_1,C_2,S_1,S_2$ are complex matrices. Then
\begin{equation}\label{sda}
M_1:=\mat{A&C_1\\0&B},\quad M_2:=\mat{0&C_2\\0&0}
,\quad
  N:=
\begin{bmatrix}
    A& 0\\0&B
  \end{bmatrix}
\end{equation}
are complex matrices too,
  $M:=
 \matt{A& C\\0&B}=M_1+M_2j$, and $\widehat S=S_1+\varepsilon S_2j$.

By \eqref{kli} and \eqref{tph},
\begin{equation}\label{whm}
(M_1+M_2j)(S_1+S_2j)=(S_1+\varepsilon S_2j)N.
\end{equation}
Applying to \eqref{whm} the injective homomorphism
\eqref{pmt}, we get
\begin{equation*}\label{huyy}
 \begin{bmatrix}
    M_1& M_2\\-\bar{M_2}&\bar{M_1}
  \end{bmatrix}
\begin{bmatrix}
    S_1& S_2\\-\bar{S_2}&\bar{S_1}
  \end{bmatrix}=
  \begin{bmatrix}
    S_1& \varepsilon S_2\\-\varepsilon \bar{S_2}&\bar{S_1}
  \end{bmatrix}
  \begin{bmatrix}
    N& 0\\0&\bar{N}
  \end{bmatrix}.
  \end{equation*}
Then
   \[
   J
 \begin{bmatrix}
    M_1& M_2\\-\bar{M_2}&\bar{M_1}
  \end{bmatrix}
\begin{bmatrix}
    S_1& S_2\\-\bar{S_2}&\bar{S_1}
  \end{bmatrix}= J
  \begin{bmatrix}
    S_1& \varepsilon S_2\\-\varepsilon \bar{S_2}&\bar{S_1}
  \end{bmatrix}
JJ
  \begin{bmatrix}
    N& 0\\0&\bar{N}
  \end{bmatrix}
  \]
with
\begin{equation}\label{vcy}
J:=\mat{I& 0\\0&\varepsilon I}
\end{equation}
gives
  \begin{equation*}\label{nuir}
 \begin{bmatrix}
    M_1& M_2\\-\varepsilon \bar{M_2}&\varepsilon \bar{M_1}
  \end{bmatrix}
\begin{bmatrix}
    S_1& S_2\\-\bar{S_2}&\bar{S_1}
  \end{bmatrix}=
  \begin{bmatrix}
    S_1& S_2\\-\bar{S_2}&\bar{S_1}
  \end{bmatrix}
  \begin{bmatrix}
    N& 0\\0&\varepsilon \bar{N}
  \end{bmatrix}.
  \end{equation*}
Therefore, the matrices
  \begin{equation*}
 \begin{bmatrix}
    M_1& M_2\\-\varepsilon \bar{M_2}&\varepsilon \bar{M_1}
  \end{bmatrix}\quad\text{and}\quad  \begin{bmatrix}
    N&0\\0&\varepsilon \bar{N}
  \end{bmatrix}
   \end{equation*}
are similar. We substitute \eqref{sda} obtaining
\[
\left[\begin{array}{cc|cc}
 A & C_1&0 &C_2\\
 0 & B&0&0\\\hline
 0\vphantom{\bar{\tilde{\widehat{a}}}} & -\varepsilon  \bar C_2&\varepsilon \bar A & \varepsilon  \bar C_1\\
 0 & 0&0 & \varepsilon \bar B
 \end{array}\right]\quad\text{and}\quad
\left[\begin{array}{cc|cc}
 A & 0&0 &0\\
 0 & B&0&0\\\hline
 0 & 0\vphantom{\bar{\tilde{\widehat{a}}}}&\varepsilon \bar A & 0\\
 0 & 0&0 & \varepsilon \bar B
 \end{array}\right],
 \]
then apply the similarity transformation given by
\begin{equation}\label{vnv}
 \begin{bmatrix}
    I&0&0&0\\0&0&I&0\\
    0&I&0&0\\0&0&0&I
  \end{bmatrix},
\end{equation}
and get the complex matrices
\[
\left[\begin{array}{cc|cc}
 A &0& C_1&C_2\\
 0 & \varepsilon \bar A&-\varepsilon \bar C_2&\varepsilon \bar C_1\\\hline
 0 & 0&B & 0\\
 0 & 0&0 & \varepsilon  \bar B
  \end{array}\right]\quad\text{and}\quad
  \left[\begin{array}{cc|cc} A & 0&0 &0\\
 0 & \varepsilon \bar A &0&0\\\hline
 0 & 0&B & 0\\
 0 & 0&0 & \varepsilon  \bar B
 \end{array}\right],
 \]
which are similar. By Roth's theorem \eqref{roh2},  the complex matrix equation
 \[ \begin{bmatrix}
    A& 0\\0&\varepsilon  \bar A
  \end{bmatrix}
   \begin{bmatrix}
    Z_1& Z_2\\Z_3&Z_4
  \end{bmatrix}-
 \begin{bmatrix}
    Z_1& Z_2\\Z_3&Z_4
  \end{bmatrix}
\begin{bmatrix}
    B& 0\\0&\varepsilon \bar B
  \end{bmatrix}=
   \begin{bmatrix}
    C_1& C_2\\-\varepsilon  \bar C_2&\varepsilon  \bar C_1
  \end{bmatrix}
  \]
has a solution.
Equating the (1,1) and (1,2) entries on both the sides, we find
\begin{equation}\label{,m}
 AZ_1 - Z_1B=C_1,\quad
 AZ_2 -\varepsilon  Z_2 \bar B=C_2.
\end{equation}
Interchanging these equations, taking their complex conjugates, and multiplying them by $\pm \varepsilon $, we obtain
\begin{equation}\label{owe}
-\varepsilon \bar A\bar Z_2 +\bar Z_2 B=-\varepsilon \bar C_2,\quad
 \varepsilon \bar A\bar Z_1 -\varepsilon \bar Z_1\bar B=\varepsilon \bar C_1.
\end{equation}
Write \eqref{,m} and \eqref{owe} in matrix form:
\[ \begin{bmatrix}
    A& 0\\0&\varepsilon \bar A
  \end{bmatrix}
   \begin{bmatrix}
    Z_1& Z_2\\- \bar Z_2&\bar Z_1
  \end{bmatrix}-
 \begin{bmatrix}
    Z_1& Z_2\\-\bar Z_2&\bar Z_1
  \end{bmatrix}
\begin{bmatrix}
    B& 0\\0&\varepsilon \bar B
  \end{bmatrix}=
   \begin{bmatrix}
    C_1& C_2\\-\varepsilon  \bar C_2&\varepsilon \bar C_1
  \end{bmatrix}.
\]
Then
 \begin{equation*}
 J\begin{bmatrix}
    A& 0\\0&\varepsilon \bar A
  \end{bmatrix}
   \begin{bmatrix}
    Z_1& Z_2\\- \bar Z_2&\bar Z_1
  \end{bmatrix}-J
 \begin{bmatrix}
    Z_1& Z_2\\-\bar Z_2&\bar Z_1
  \end{bmatrix}J J
\begin{bmatrix}
    B& 0\\0&\varepsilon \bar B
  \end{bmatrix}=J
   \begin{bmatrix}
    C_1& C_2\\ -\varepsilon \bar C_2&\varepsilon  \bar C_1
  \end{bmatrix}
  \end{equation*}
with $J$ defined in \eqref{vcy} gives
   \begin{equation*}
 \begin{bmatrix}
    A& 0\\0&\bar A
  \end{bmatrix}
   \begin{bmatrix}
    Z_1& Z_2\\- \bar Z_2&\bar Z_1
  \end{bmatrix}-
 \begin{bmatrix}
    Z_1& \varepsilon Z_2\\-\varepsilon \bar Z_2&\bar Z_1
  \end{bmatrix}
\begin{bmatrix}
    B& 0\\0&\bar B
  \end{bmatrix}=
   \begin{bmatrix}
    C_1& C_2\\ -\bar C_2& \bar C_1
  \end{bmatrix}.
  \end{equation*}
Due to the homomorphism \eqref{pmt},
the quaternion matrix $Z_1+Z_2j$ is a solution of $AX-\widehat XB=C$.
\medskip

\emph{Case 2: $A$ and $B$ are quaternion matrices}.
Let $X=PYQ$, where $P$ and $Q$ are some nonsingular quaternion matrices and $Y$ is a new unknown matrix.
Substituting $X=PYQ$ into
${AX-\widehat XB=C}$, we get
\[{APYQ-\widehat P\widehat Y\widehat QB=C}.\]
Multiply it by $\widehat P^{-1}$ on the left and by $Q^{-1} $ on the right:
\begin{equation}\label{plo}
\widehat P^{-1}AP\cdot Y- \widehat Y \cdot \widehat QBQ^{-1}=\widehat P^{-1}CQ^{-1}.
\end{equation}
Choose $P$ and $Q$ such that $\widehat P^{-1}AP$ and $\widehat QBQ^{-1}$ are real matrices, which is possible due to \eqref{nvn}.

By Case 1, \eqref{plo} has a solution if
\[
\mat{\widehat P^{-1}AP&\widehat P^{-1}CQ^{-1}\\0&
\widehat QBQ^{-1}}\quad\text{and}\quad
\mat{\widehat P^{-1}AP&0\\0&
\widehat QBQ^{-1}}
\]
are ${\wedge}$-similar. These matrices are ${\wedge}$-similar since they are equal to
\begin{equation*}\label{nbv}
\mat{{\widehat P}^{-1}& 0\\0&\widehat Q}
\mat{A& C\\0&B}
\mat{P& 0\\0&Q^{-1}}\quad\text{and}\quad
\mat{{\widehat P}^{-1}& 0\\0&\widehat Q}
  \mat{A& 0\\0&B}
\mat{P& 0\\0&Q^{-1}},
\end{equation*}
which are ${\wedge}$-similar to $\matt{A& B\\0&C}$ and
$\matt{A& 0\\0&B}$, which
are ${\wedge}$-similar by \eqref{tph}. Thus, \eqref{plo} has a solution, and so $AX-\widehat XB=C$ has a solution too.
\end{proof}


\section{Roth's theorem for the quaternion matrix equation $ X-A\widehat{X}B=C$}

The following theorem is the quaternion version of Wimmer's theorem \eqref{nrg}.

\begin{theorem}\label{iii2}
Let $h\mapsto \hat h$ be an involutive automorphism of the skew field of quaternions.
The quaternion matrix equation
$
X-A\widehat{X}B=C
$ has a solution if and only if
there exist nonsingular quaternion matrices $S$ and $R$ such that
\begin{equation}\label{tyu}
  \begin{bmatrix}
   A&C\\0&I
  \end{bmatrix}R
=
\widehat{S}
  \begin{bmatrix}
   A&0\\0&I
  \end{bmatrix}, \qquad
  \begin{bmatrix}
   I&0\\0&B
  \end{bmatrix}R=
  S
  \begin{bmatrix}
   I&0\\0&B
  \end{bmatrix}.
\end{equation}
\end{theorem}

\begin{proof}
$ \Longrightarrow$. \ If $X$ is a solution of $
X-A\widehat{X}B=C
$, then \eqref{tyu} holds with $R=\matt{I&\widehat XB\\0&I}$ and $S=\matt{I&\widehat X\\0&I}$ since
\begin{equation}\label{dss}
\begin{split}
  \begin{bmatrix}
    A&X-A\widehat XB\\0&I
  \end{bmatrix}
   \begin{bmatrix}
    I&\widehat XB\\0&I
  \end{bmatrix}
             &=
    \begin{bmatrix}
    A&{X}\\0&I
  \end{bmatrix}
  =
\begin{bmatrix}
    I&X\\0&I
  \end{bmatrix}
  \begin{bmatrix}
    A&0\\0&I
  \end{bmatrix},
                        \\
  \begin{bmatrix}
    I&0\\0&B
  \end{bmatrix}
   \begin{bmatrix}
    I&\widehat XB\\0&I
  \end{bmatrix}
&=  \begin{bmatrix}
    I&\widehat XB\\0&B
  \end{bmatrix}=
\begin{bmatrix}
    I&\widehat X\\0&I
  \end{bmatrix}
  \begin{bmatrix}
    I&0\\0&B
  \end{bmatrix}.
       \end{split}
\end{equation}

\medskip

$ \Longleftarrow$. \
Due to Lemma \ref{lih2}, we suppose that the automorphism $h\mapsto\hat h$ is of the form \eqref{kli}.
Let \eqref{tyu} hold.
\medskip

\emph{Case 1: $A$ and $B$ are complex matrices}.
Write \[C=C_1+C_2j,\quad S=S_1+S_2j,\quad R = R_1+R_2j,\] where $C_1,C_2,S_1,S_2,R_1,R_2$ are complex matrices. Then
\begin{equation}\label{jke}
 M_1:=
 \begin{bmatrix}
    A& C_1\\0&I
  \end{bmatrix},\quad
 M_2:=
 \begin{bmatrix}
    0& C_2\\0&0
  \end{bmatrix},\quad
  N:=
\begin{bmatrix}
    A& 0\\0&I
  \end{bmatrix},\quad
   L:=
\begin{bmatrix}
    I& 0\\0&B
  \end{bmatrix}
 \end{equation}
are complex matrices too,
$
M:=\matt{A& C\\0&I}=M_1+M_2j
$, and $\widehat S=S_1+\varepsilon S_2j$. By \eqref{tyu},
$  MR=\widehat{S}N$ and $LR=SL.
$ Applying to them the injective homomorphism
\eqref{pmt}, we get
  \begin{align}\nonumber
 \begin{bmatrix}
    M_1& M_2\\-\bar M_2&\bar M_1
  \end{bmatrix}
\begin{bmatrix}
    R_1& R_2\\-\bar R_2 &\bar R_1
  \end{bmatrix}&=
  \begin{bmatrix}
    S_1& \varepsilon S_2\\-\varepsilon \bar S_2 &\bar S_1
  \end{bmatrix}
  \begin{bmatrix}
    N& 0\\0&\bar{N}
  \end{bmatrix},\\
   \label{fgw}
 \begin{bmatrix}
    L& 0\\0&\bar{L}
  \end{bmatrix}
\begin{bmatrix}
    R_1& R_2\\-\bar R_2&\bar R_1
  \end{bmatrix}&=
  \begin{bmatrix}
    S_1& S_2\\-\bar S_2 &\bar S_1
  \end{bmatrix}
  \begin{bmatrix}
    L& 0\\0&\bar{L}
  \end{bmatrix}.
  \end{align}

By the first equation,
   \begin{equation*}
   J
 \begin{bmatrix}
    M_1& M_2\\-\bar M_2 &\bar M_1
  \end{bmatrix}
\begin{bmatrix}
    R_1& R_2\\-\bar R_2 &\bar R_1
  \end{bmatrix}= J
  \begin{bmatrix}
    S_1& \varepsilon S_2\\-\varepsilon \bar S_2 &\bar S_1
  \end{bmatrix}
  JJ
  \begin{bmatrix}
    N& 0\\0&\bar{N}
  \end{bmatrix}
  \end{equation*}
with $J$ defined in \eqref{vcy}, which gives
  \begin{equation*}
 \begin{bmatrix}
    M_1& M_2\\-\varepsilon \bar M_2&\varepsilon \bar M_1
  \end{bmatrix}
\begin{bmatrix}
    R_1& R_2\\-\bar R_2 &\bar R_1
  \end{bmatrix}=
  \begin{bmatrix}
    S_1& S_2\\-\bar S_2 &\bar S_1
  \end{bmatrix}
  \begin{bmatrix}
    N& 0\\0&\varepsilon \bar{N}
  \end{bmatrix}.
  \end{equation*}
This equation and \eqref{fgw} ensure that the matrices in the pairs
\[
\left(
\begin{bmatrix}
    M_1& M_2\\-\varepsilon \bar{M_2}&\varepsilon \bar{M_1}
  \end{bmatrix},
\begin{bmatrix}
    L& 0\\0&\bar{L}
  \end{bmatrix}
\right)
\quad\text{and}\quad
\left(
\begin{bmatrix}
    N& 0\\0&\varepsilon \bar{N}
  \end{bmatrix},
\begin{bmatrix}
    L& 0\\0&\bar{L}
  \end{bmatrix}
\right)\]
are simultaneously equivalent. Substituting \eqref{jke}, we get
\[
\left(\left[
\begin{array}{cc|cc}
 A & C_1&0 &C_2\\
 0 & I&0&0\\\hline
 0\vphantom{\bar{\tilde{\widehat{a}}}} & -\varepsilon \bar C_2&\varepsilon \bar A & \varepsilon \bar  C_1\\
 0 & 0&0 & \varepsilon I
 \end{array}\right],\ \left[
\begin{array}{cc|cc}
 I & 0&0 &0\\
 0 & B&0&0\\\hline
 0 & 0& I& 0\\
 0 & 0&0 & \bar B
 \end{array}\right]
 \right)
 \end{equation*}
and
\[
\left(
\left[
\begin{array}{cc|cc}
 A &0&0 &0\\
 0 & I&0&0\\\hline
 0\vphantom{\bar{\tilde{\widehat{a}}}} & 0&\varepsilon \bar A & 0\\
 0 & 0&0 & \varepsilon I
 \end{array}\right],\
\left[
\begin{array}{cc|cc}
 I & 0&0 &0\\
 0 & B&0&0\\\hline
 0 & 0& I& 0\\
 0 & 0&0 & \bar B
 \end{array}\right]
 \right).
 \]
Multiplying them by
 \[
  \begin{bmatrix}
    I&0&0&0\\0&0&I&0\\
    0&I&0&0\\0&0&0&\varepsilon
  \end{bmatrix}
 \]
 on the left and by \eqref{vnv} on the right,
we obtain the pairs
\[
\left(
\left[
\begin{array}{cc|cc}
 A &0 &C_1 & C_2\\
 0 &\varepsilon \bar A &-\varepsilon \bar C_2& \varepsilon\bar C_1\\\hline
 0 & 0&I &0 \\
 0 & 0&0 & I
 \end{array}\right],\ \left[
\begin{array}{cc|cc}
 I & 0&0 &0\\
 0 & I&0&0\\\hline
 0 & 0& B& 0\\
 0 & 0&0 &\varepsilon \bar B
\end{array}\right]
 \right)
 \]
and
\[
\left(
\left[
\begin{array}{cc|cc}
 A &0&0 &0\\
 0 & \varepsilon  \bar A&0&0\\\hline
 0 & 0& I& 0\\
 0 & 0&0 & I
 \end{array}\right],\ \left[
\begin{array}{cc|cc}
 I & 0&0 &0\\
 0 & I&0&0\\\hline
 0 & 0& B& 0\\
 0 & 0&0 &\varepsilon \bar B
\end{array}\right]
 \right),
 \]
whose matrices are simultaneously equivalent.
By Wimmer's criterion \eqref{nrg}, the complex matrix equation
 \[    \begin{bmatrix}
    Z_1& Z_2\\Z_3&Z_4
  \end{bmatrix}-
  \begin{bmatrix}
    A& 0\\0&\varepsilon \bar A
  \end{bmatrix}
 \begin{bmatrix}
    Z_1& Z_2\\Z_3&Z_4
  \end{bmatrix}
\begin{bmatrix}
    B& 0\\0&\varepsilon \bar B
  \end{bmatrix}=
   \begin{bmatrix}
    C_1& C_2\\ -\varepsilon \bar C_2&\varepsilon \bar C_1
  \end{bmatrix}
  \]
has a solution.
Equating the (1,1) and (1,2) entries on both the sides, we find
\begin{equation}\label{llk}
 Z_1 - AZ_1B=C_1,\quad
 Z_2 -\varepsilon  AZ_2 \bar B=C_2.
\end{equation}
Interchanging these equations, taking their complex conjugates, and multiplying them by $\pm \varepsilon $, we obtain
\begin{equation}\label{aabb}
 -\varepsilon \bar{Z}_2 +  \bar{A}\bar{Z}_2{B}=-\varepsilon \bar{C}_2,
  \qquad
\varepsilon \bar{Z}_1 - \varepsilon \bar{A}\bar {Z}_1\bar{B}=\varepsilon \bar{C}_1.
\end{equation}
Write \eqref{llk} and \eqref{aabb} in matrix form:
  \[
   \begin{bmatrix}
    Z_1& Z_2\\-\varepsilon  \bar Z_2&\varepsilon\bar Z_1
  \end{bmatrix}-
   \begin{bmatrix}
    A& 0\\0&\varepsilon  \bar A
  \end{bmatrix}
 \begin{bmatrix}
    Z_1&Z_2\\-\varepsilon \bar Z_2&\varepsilon \bar Z_1
  \end{bmatrix}
\begin{bmatrix}
    B& 0\\0&\varepsilon \bar B
  \end{bmatrix}=
   \begin{bmatrix}
    C_1& C_2\\ -\varepsilon \bar C_2&\varepsilon  \bar C_1
  \end{bmatrix}.
    \]
Then
  \[
  J
   \begin{bmatrix}
    Z_1& Z_2\\-\varepsilon  \bar Z_2&\varepsilon\bar Z_1
  \end{bmatrix}-J
   \begin{bmatrix}
    A& 0\\0&\varepsilon  \bar A
  \end{bmatrix}
 \begin{bmatrix}
    Z_1& Z_2\\-\varepsilon \bar Z_2&\varepsilon \bar Z_1
  \end{bmatrix}JJ
\begin{bmatrix}
    B& 0\\0&\varepsilon \bar B
  \end{bmatrix}=J
   \begin{bmatrix}
    C_1& C_2\\ -\varepsilon \bar C_2&\varepsilon  \bar C_1
  \end{bmatrix}
  \]
with $J$ defined in \eqref{vcy} gives
   \[
   \begin{bmatrix}
    Z_1& Z_2\\- \bar Z_2&Z_1
  \end{bmatrix}-
   \begin{bmatrix}
    A& 0\\0&\bar A
  \end{bmatrix}
 \begin{bmatrix}
    Z_1& \varepsilon Z_2\\-\varepsilon \bar Z_2&\bar Z_1
  \end{bmatrix}
\begin{bmatrix}
    B& 0\\0&\bar B
  \end{bmatrix}=
   \begin{bmatrix}
    C_1& C_2\\ -\bar C_2& \bar C_1
  \end{bmatrix}.
  \]
Due to the homomorphism \eqref{pmt},
the quaternion matrix $Z_1+Z_2j$ is a solution of $ X-A\widehat{X}B=C$.
\medskip

\emph{Case 2: $A$ and $B$ are quaternion matrices}.
Let $X=PYQ$, where $P$ and $Q$ are some nonsingular quaternion matrices  and $Y$ is a new unknown matrix.
Substituting $X=PYQ$ into
${X-A\widehat{X}B=C}$, we get
\[{P}{Y}{Q}- A\widehat P\widehat Y\widehat QB=C.\]
Multiply it by $P^{-1}$ on the left and by $Q^{-1} $ on the right:
\begin{equation}\label{mju}
{Y}-P^{-1}A\widehat P \cdot\widehat Y \cdot \widehat QB{ Q}^{-1}
= P^{-1}C Q^{-1}.
\end{equation}
Choose $P$ and $Q$ such that $ P^{-1}A\widehat P$ and $\widehat QBQ^{-1}$ are complex matrices, which is possible due to \eqref{nvn}.

By Case 1, \eqref{mju} has a solution if
\begin{equation}\label{vfi}
\begin{split}
  \mat{P^{-1}A\widehat P&P^{-1}CQ^{-1}\\0&I}R'
&= \widehat S'\mat{P^{-1}A\widehat P& 0\\0&I}
              \\
\mat{I& 0\\0&\widehat QBQ^{-1}}R'&=
S'\mat{I& 0\\0&\widehat QBQ^{-1}}
\end{split}
\end{equation}
for some nonsingular quaternion matrices $R'$ and $S'$. Write these equations
in the form:
\begin{align*}
\mat{P^{-1}& 0\\0&Q}
\mat{A& C\\0&I}
\mat{\widehat P& 0\\0&Q^{-1}}R'
   & =\widehat S'
\mat{P^{-1}& 0\\0&Q}
\mat{A& 0\\0&I}
\mat{\widehat P& 0\\0& Q^{-1}},
           \\
\mat{\widehat P^{-1}& 0\\0&\widehat Q}
\mat{I& 0\\0&B}
\mat{\widehat P& 0\\0&Q^{-1}}
 R' &=
S'\mat{\widehat P^{-1}& 0\\0&\widehat Q}
\mat{I& 0\\0&B}
\mat{\widehat P& 0\\0&Q^{-1}}.
\end{align*}
Setting
\begin{equation}\label{las}
R:=\mat{\widehat P& 0\\0&Q^{-1}}
 R'\mat{\widehat P^{-1}& 0\\0&Q},\quad
S:=\mat{\widehat P& 0\\0&\widehat Q^{-1}}
S'\mat{\widehat P^{-1}& 0\\0&\widehat Q},
\end{equation}
we get \eqref{tyu}.

Thus, if the condition \eqref{tyu} holds for some matrices $R$ and $S$, then we can define $R'$ and $S'$ from \eqref{las} and obtain the equalities \eqref{vfi}. They ensure the solvability of \eqref{mju},  and so the solvability of $X-A\widehat XB=C$.
\end{proof}

\section*{Acknowledgements}
V. Futorny is
supported in part by  CNPq grant (301320/2013-6) and by
FAPESP grant (2014/09310-5).
This work was done during the visit of V.V.~Sergeichuk to the University of S\~ao Paulo;
he is grateful to the university for hospitality and the FAPESP
for financial support (grant 2015/05864-9).
The authors wish to express their gratitude to the referee for his comments and suggestions.

\end{document}